\documentclass{article}

\usepackage[a4paper, margin=2.5cm]{geometry}
\usepackage[leqno,intlimits]{amsmath}
\usepackage{amssymb}
\usepackage{amsthm}
\usepackage{hyperref}
\newtheorem{tm}{Theorem}[section]
\newtheorem{pr}[tm]{Proposition}

\numberwithin{equation}{section}
\newcommand*{\pf}{\mathfrak p}
\newcommand*{\qf}{\mathfrak q}
\newcommand*{\hop}{\bigskip\noindent}
\newcommand*{\e}[1]{\text{\rm e}^{#1}}
\newcommand*{\rup}{r_+^\uparrow}
\newcommand*{\rnup}{r_0^\uparrow}
\newcommand*{\rdn}{r^\downarrow}
\newcommand*{\Om}{\Omega}
\newcommand*{\om}{\omega}
\newcommand*{\ze}{\zeta}
\newcommand*{\vp}{\varphi}
\newcommand*{\vr}{\varrho}
\newcommand*{\te}{\theta}
\newcommand*{\la}{\lambda}
\newcommand*{\be}{\beta}
\newcommand*{\si}{\sigma}
\newcommand*{\Rb}{\mathbb R}
\newcommand*{\Zb}{\mathbb Z}
\newcommand*{\ac}{\mathfrak a}
\newcommand*{\bc}{\mathfrak b}
\newcommand*{\Hc}{\mathcal H}
\newcommand*{\Ev}{{\bf E}}
\newcommand*{\un}[1]{\underline{#1}}
\newcommand*{\di}{\,\text{\rm d}}

\begin{document}

\title{q-zero range has random walking shocks}

\author{M\'arton Bal\'azs\footnotemark[1], Lewis Duffy\footnotemark[1] \ and Dimitri Pantelli\footnotemark[1]\thanks{University of Bristol, UK. M.B.\ was partially supported by EPSRC's EP/R021449/1 Standard Grant.}}

\maketitle

\begin{abstract}
 \dots but no other surprises in the zero range world. We check all nearest neighbour 1-dimensional asymmetric zero range processes for random walking product shock measures as demonstrated already for a few cases in the literature. We find the totally asymmetric version of the celebrated q-zero range process as the only new example besides an already known model of doubly infinite occupation numbers and exponentially increasing jump rates. We also examine the interaction of shocks, which appears somewhat more involved for q-zero range than in the already known cases.
\end{abstract}


\hop

\section{Introduction}

Since the early investigations of their hydrodynamic scaling limits (e.g., Rezakhanlou \cite{hl}), it has been clear that several 1-dimensional asymmetric interacting particle systems develop shocks; moving discontinuities of the rescaled density profile. Such shocks are very well understood on the hydrodynamic PDE level, but less so in the stochastic model itself. This microscopic phenomenon has been investigated by several authors, we give a brief summary of some of the important steps in the area.

Second class particles are coupling objects on the level of particles that trace the difference between two coupled instances of a particle model with slightly differing initial conditions. We describe the precise coupled dynamics for our models in Section \ref{sc:zrp}. Briefly, second class particles use up the difference of jump rates between the coupled pair of models. Their connections to shocks have a long history of exploration, we only mention a few instances of this here. Derrida et al.\ \cite{dls} investigated stationary distributions in the asymmetric simple exclusion process (ASEP) as seen by a \emph{second class particle}, and realised that their description greatly simplifies under certain condition on the jump of the densities in the shock. Namely, the stationary distribution becomes a product of Bernoulli distributions with different constant densities on the two sides of the second class particle. Bal\'azs \cite{valak} reproduced this result in the exponential bricklayers process (EBLP). Next, Belitsky and Sch\"utz \cite{qse} were able to transfer the result in a fixed frame of reference by observing that under the same condition on shock densities, the ASEP evolves a shock product distribution into translations of this same distribution, moreover the coefficients are the transition probabilities of a drifted simple random walk. We therefore call this phenomenon \emph{random walking shocks}, and will give its precise formulation in Section \ref{sc:rws}.

Belitsky and Sch\"utz \cite{qse} could also handle multiple shocks, exploring the interaction between them. Random walking shocks of ASEP interact like another ASEP, except the jump rates of these shock-walkers change with the rank as we consider them from left to right. This results in an attraction between the shock-random walkers, making them appear as one large shock after rescaling. Bal\'azs \cite{sokvalak} again reproduced this result in the EBLP and noted that, rather than excluding each other, in this model the random walkers almost act independently except for their rank-based jump rates. This again makes the walkers attract each other and a group of microscopic shocks appear as one large shock in the hydrodynamic scale. A precise statement for multiple random walking shocks follows in Section \ref{sc:rws}.

The random walking shock results had not used the second class particle until Bal\'azs et al.\ \cite{rwshscp} reproduced both random walking shock results (for ASEP and EBLP) with a second class particle added to the position where the densities jump. This on one hand simplified the arguments, on the other hand explained both the stationarity results from the viewpoint of the second class particle, and the random walking shock measures that first came without the second class particle. They also added a generalised exponential zero range process; a doubly infinite extension of zero range processes with an exponential form of the jump rates. This process came naturally as it has half the dynamics of EBLP. Certain other processes (e.g., \cite{rwshscp,Rakos2004}) are known to possess the random walking shock property, to the best knowledge of the authors none of which belong to the zero range family.

Stochastic duality, if discovered for a model, is a very powerful tool for the analysis of Markov processes. Given a process \(X(t)\) with state space \(\Om_X\), to establish stochastic duality one needs to find a dual process \(Y(t)\) on a state space \(\Om_Y\) and a duality function \(H\,:\,\Om_X\times\Om_Y\to\Rb\) such that for all \(t\ge0\) and \(x\in\Om_X\), \(y\in\Om_Y\),
\[
 \Ev_xH\bigl(X(t),\,y\bigr)=\Ev_yH\bigl(x,\,Y(t)\bigr).
\]
Here \(\Ev_x\) refers to expectation w.r.t.\ the dynamics of process \(X(\cdot)\) started from state \(X(0)=x\), and the meaning of \(\Ev_y\) is analogous. Often the model \(Y(\cdot)\) turns out to be simpler than \(X(\cdot)\), in which case various properties of this latter can be deduced from such a duality. A general account on duality for interacting particle systems can be found in Liggett's book \cite{ips}. The function \(H\) and the process \(Y(\cdot)\) are in general not easy to find. One of the early works discussing these in an asymmetric scenario is due to Sch\"utz \cite{schutz_dual_97}. Such duality properties made an important role in the proofs of \cite{qse}. Recently, systematic work on duality for particle systems has been performed, we cite Groenevelt \cite{groe_o_dual_18}, Redig and Sau \cite{red_sau_fact_dual_18,red_sau_dual_eigenf_18}, and references therein. Multispecies models (i.e., with second class particles) also possess duality properties as demonstrated by Belitsky and Sch\"utz \cite{beli_schutz_selfd_2comp_15} and later generalized by Kuan \cite{kuan_multi_aqj_18}. The importance of these developments for our context is highlighted in Belitsky and Sch\"utz \cite{bel_schutz_selfd_shockdyn_18}, where it becomes apparent that random walking shocks in ASEP, including the second class particle, are directly related to stochastic duality functions. The key observation is that certain duality functions can act as Radon-Nikodym derivatives between well-known stationary distributions and random walking shock measures. The dual process is (the time-reversal of) the random walkers themselves.

In this work we perform a systematic search within 1-dimensional nearest neighbour attractive zero range processes and find the totally asymmetric q-zero range (q-TAZRP) as the only new model that has the random walking shocks property. This model is remarkable for its algebraic (and now, probabilistic) properties. We do not use dualities, however in light of recent developments our results give strong indications regarding the nature of dualities behind q-TAZRP. Indeed, some details already emerged in private communications with Frank Redig -- further investigations will be part of future work.

We believe that our result and purely probabilistic approach will also nicely complement the more analytic research that has been done on q-TAZRP, some of which we briefly mention next.

Some of the early mentions of the model is by Povolotsky \cite{povo_bethe_zrp}, where it is found as one suitable for Bethe ansatz. Remarkably, q-TAZRP is one of the examples that satisfy the \emph{microscopic concavity property}, hence Kardar-Parisi-Zhang (KPZ)-scaling of current fluctuations could be proved using purely probabilistic methods in \cite{unipq3}. It also arises as the inter-particle distance process for q-TASEP, distortion of totally asymmetric simple exclusion and a very celebrated particle system in the integrable probability literature. Duality and determinantal formulas allow exact calculations in Borodin et al.\ \cite{borodin_corwin_sasamoto_duality} that show connections of q-TASEP to the celebrated KPZ equation. It has also been shown that q-TASEP arises as certain projection of McDonald processes, see Borodin and Corwin \cite{bo_co_mcdonald}. Bethe ansatz and integrable probability techniques allowed Korhonen and Lee \cite{kor_lee_trpr_l_qzrp} to derive exact formulas for certain transition probabilities in q-TAZRP. Ferrari and Vet\H o \cite{ferr_veto_tr-w_qtasep} and Barraquand \cite{barr_phtr_qtasep} proved Tracy-Widom limits for current fluctuations in q-TASEP, while recently Imamura and Sasamoto \cite{ima_sasa_fluct_qtasep} proved KPZ-related scaling limits of the tagged particle distribution. Finally, we mention Barraquand \cite{barr_phtr_qtasep} and Lee, Wang \cite{lee_wang_part_qzrp_sitedep} as investigations for the non-homogeneous version of the dynamics.

In this article we stick to simple probabilistic methods, to be more precise we apply the arguments of \cite{rwshscp} to carry out our search. For this reason at several points we simply refer to \cite{rwshscp} rather than repeat all arguments from therein. Our notation fully conforms \cite{rwshscp}. We start by recalling parts of the zero range framework in Section \ref{sc:zrp}, then state and prove our results in Sections \ref{sc:rws} and \ref{sc:pf}.

\section{Zero range process}\label{sc:zrp}

We consider 1-dimensional asymmetric zero range processes with site occupation numbers \(\om_i\) in the range either \(I=\Zb\) or \(I=\Zb^+=\{0,\,1,\,2,\,\dots\}\). Due to the zero range-type dynamics, these numbers cannot be bounded from above, while any finite lower bound can be transformed into 0 by a shift in the interpretation and the rate function of the jumps. Particles jump from site \(i\) with rate \(g(\om_i)\) choosing the right or left neighbouring site according to probabilities \(0<p\le1\) and \(1-p\). The rate function \(g\) is non-negative and we assume it to be non-decreasing and not the constant function throughout \(\Zb\). To preserve the state space, we also assume \(g(0)=0\) if \(I=\Zb^+\). The state space of the process is \(\Om=I^\Zb\), and the formal infinitesimal generator is
\[
 L\vp(\un\om)=\sum_{i=-\infty}^\infty pg(\om_i)\bigl[\vp\bigl(\un\om^{i,\,i+1}\bigr)-\vp(\un\om)\bigr]+ (1-p)g(\om_i)\bigl[\vp\bigl(\un\om^{i,\,i-1}\bigr)-\vp(\un\om)\bigr]
\]
with notation
\begin{equation}
 \bigl(\un\om^{i,\,j}\bigr)_k:\,=\left\{
  \begin{aligned}
   &\om_k-1,&&\text{for }k=i,\\
   &\om_k+1,&&\text{for }k=j,\\
   &\om_k,&&\text{otherwise.}
  \end{aligned}
 \right.\label{eq:ijch}
\end{equation}
Under appropriate initial data, the dynamics of this family has been constructed for \(g\)'s of bounded increments by Andjel \cite{and}, and of the very similar \emph{bricklayers processes} in the totally asymmetric case (i.e., \(p=1\)) up to exponentially growing \(g\)'s in Bal\'azs et al.\ \cite{exists}. Non-decreasing \(g\)'s make the model attractive, and it becomes a member of the family of misanthrope processes as described in \cite{rwshscp} with \(f=g\) from therein. The product measure \(\un\mu^\te:\,=\bigotimes_{i\in\Zb}\mu^\te\) with marginal
\begin{equation}
 \mu^\te(z)=\left\{
  \begin{aligned}
   &0,&&\text{if }z\notin I,\\
   &\frac1{Z(\te)}\cdot\frac{\e{\te z}}{g(z)!},&&\text{otherwise}
  \end{aligned}
 \right.\label{eq:mum}
\end{equation}
on site \(i\in\Zb\) is stationary for \(\un\te<\te<\bar\te\) with
\[
 \begin{aligned}
  \bar\te:&=\lim_{z\to\infty}\ln(g(z)),\quad\text{and}\\
  \un\te:&=\left\{
   \begin{aligned}
    &\lim_{z\to-\infty}\ln(g(z)),&&\text{if }\ I=\Zb,\\
    &-\infty,&&\text{if }\ I=\Zb^+.
   \end{aligned}
  \right.
 \end{aligned}
\]
This is the range that makes the sum of \(\mu^\te(z)\) convergent. Here we used the usual definition
\[
 g(z)!:\,=\prod_{y=1}^zg(y)\quad(z>0),\qquad g(0)!:\,=1,\qquad g(z)!:\,=\frac{1}{\prod\limits_{y=z+1}^0g(y)}\quad(z<0).
\]
Expectation w.r.t.\ this marginal will be denoted by \(\Ev^\te\).

We briefly comment on hydrodynamics of the model. Define the \emph{density of particles} under the product measure \(\un\mu^\te\) as \(\vr(\te):\,=\Ev^\te\om_i\). It is readily checked from the definition that this is a strictly increasing function; its inverse is denoted by \(\te(\vr)\). With the \emph{hydrodynamics flux}, which is the expected jump rate in stationarity at density \(\vr\):
\[
 \Hc(\vr):\,=\Ev^{\te(\vr)}\bigl(pg(\om_i)-(1-p)g(\om_i)\bigr),
\]
the density profile is expected to satisfy
\begin{equation}
 \partial_T\rho(T,\,X)+\partial_X\Hc(\rho(T,\,X))=0\label{eq:hcl}
\end{equation}
with the rescaled time and space variables \(T=\varepsilon t\) and \(X=\varepsilon i\) as \(\varepsilon\to0\). This has in fact been proved for several cases of the rates \(g\), see e.g.\ Rezakhanlou \cite{hl} or Bahadoran, Guiol, Ravishankar and Saada \cite{bagurasa} for details.

It has also been shown in \cite{convex} that \(\Hc\) is convex, resp.\ concave, if \(g\) is so. In both cases, similarly to the inviscid Burgers equation, \eqref{eq:hcl} develops discontinuous shock solutions. These are increasing jumps for concave \(\Hc\) and decreasing jumps for convex \(\Hc\). Shocks of the PDE \eqref{eq:hcl} have a well-defined velocity, described by the Rankine-Hugoniot formula
\begin{equation}
 V=\frac{\Hc(\la)-\Hc(\vr)}{\la-\vr}\label{eq:rhv}
\end{equation}
for the case of densities \(\la\) and \(\vr\) on the two sides of the shock.

The subject of this paper is the microscopic structure of shocks, and the way we handle this is via inserting \emph{second class particles} as markers for the shock positions. These are objects arising from \emph{basic coupling} of the models. We briefly summarize this classical object for zero range processes. We consider two realisations \(\un\om(t)\) and \(\un\ze(t)\) of the same zero range dynamics with initial states that satisfy \(\om_i(0)\le\ze_i(0)\) for each \(i\in\Zb\). The formal infinitesimal generator for the joint evolution of this pair makes sure that
\begin{itemize}
 \item the order \(\om_i(t)\le\ze_i(t)\) is a.s.\ kept for any time \(t\ge0\),
 \item marginally \(\un\om(t)\) follows the zero range dynamics,
 \item marginally \(\un\ze(t)\) also follows the zero range dynamics.
\end{itemize}
This generator on suitable functions \(\vp\) of the pair \((\un\om,\,\un\ze)\) is given by
\begin{equation}
 \begin{aligned}
  (L\vp)(\un\om,\,\un\ze)=\sum_{i=-\infty}^\infty\Bigl\{&pg(\om_i)\cdot\bigl[\vp(\un\om^{i,i+1},\,\un\ze^{i,i+1})-\vp(\un\om,\,\un\ze)\bigr]\\
   +\bigl[&pg(\ze_i)-pg(\om_i)\bigr]\cdot\bigl[\vp(\un\om,\,\un\ze^{i,i+1})-\vp(\un\om,\,\un\ze)\bigr]\\
   +\phantom{\bigl[}&(1-p)g(\om_i)\cdot\bigl[\vp(\un\om^{i,i-1},\,\un\ze^{i,i-1})-\vp(\un\om,\,\un\ze)\bigr]\\
   +\bigl[&(1-p)g(\ze_i)-(1-p)g(\om_i)\bigr]\cdot\bigl[\vp(\un\om,\,\un\ze^{i,i-1})-\vp(\un\om,\,\un\ze)\bigr]\Bigr\}.
 \end{aligned}\label{eq:cgen}
\end{equation}
The rates in this generator are non-negative due to \(g\) being non-decreasing. We say that there are \(\ze_i(t)-\om_i(t)\) many second class particles at site \(i\) at time \(t\). It is easy to see that the coupled dynamics conserves the number of second class particles.

\section{Random walking shocks}\label{sc:rws}

We follow the ideas and methods in \cite{rwshscp}; for completeness, we repeat the definitions in our context. For comparison, the rates from \cite{rwshscp} now read
\begin{equation}
 \pf(y,\,z)=pg(y)\quad\text{and}\quad\qf(y,\,z)=(1-p)g(z).\label{eq:pqdef}
\end{equation}
We start with a single shock, where we only have one second class particle in the system. The shock measure we consider in this case is a product (for sites) probability distribution on the pair \((\un\om,\,\un\ze)\):
\[
 \un\nu_j:\,=\bigotimes_{i<j}\nu^\te\bigotimes_{i=j}\hat\nu\bigotimes_{i>j}\nu^\si,
\]
where the indices represent sites of \(\Zb\), \(\nu^\te\) is the stationary measure \eqref{eq:mum} concentrated on the diagonal of the single-site coupling state space \(I\times I\):
\[
 \nu^\te(y,\,z)=\left\{
  \begin{aligned}
   &\mu^\te(y),&&\text{if }y=z,\\
   &0,&&\text{if }y\ne z,
  \end{aligned}
 \right.
\]
\(\hat\nu\) is an arbitrary probability measure on \(I\times I\) with exactly one second class particle:
\[
 \hat\nu(y,\,z)=\left\{
  \begin{aligned}
   &\hat\mu(y),&&\text{if }z=y+1,\\
   &0,&&\text{otherwise,}
  \end{aligned}
 \right.
\]
and \(\te\),\,\(\si\) are the parameters of the shock measure on the left- and right-hand sides, respectively. This distribution describes a two-sided stationary distribution for sites \(i\ne j\) with constant densities and no second class particles, and one second class particle in a yet unknown distribution \(\hat\mu\) for \(\om_j=\ze_j-1\) at site \(j\).

We say that this distribution satisfies the \emph{random walking shock property}, if its evolution according to the coupled dynamics \eqref{eq:cgen} obeys
\begin{equation}
 \frac{\di}{\di t}\un\nu_jS(t)\Bigr|_{t=0}=P\cdot[\un\nu_{j+1}-\un\nu_j]+Q\cdot[\un\nu_{j-1}-\un\nu_j]\label{eq:main}
\end{equation}
with non-negative numbers \(P\) and \(Q\) that can depend on parameters \(\te\) and \(\si\). This says that \(\un\nu_j\) evolves into a linear combination of translations of this same measure, and the coefficients in this linear combination agree with the transition probabilities of a simple random walk of right jump rate \(P\) and left jump rate \(Q\).

Here is our first result of a systematic search within the class of zero range processes above for random walking shocks.
\begin{tm}\label{tm:single}
 The random walking shock property holds for zero range processes as described above in exactly the following cases.
 \begin{enumerate}
  \item\label{it:iw} \(I=\Zb^+\), \(g(y)=\ac y\), \(\ac>0\), \(\te=\si\), \(\hat\mu=\mu^\te=\mu^\si\); then \(P=p\ac\) and \(Q=(1-p)\ac\).
  \item\label{it:qzrp} \(I=\Zb^+\), \(g(y)=\ac\cdot\bigl(1-\e{-\be y}\bigr)\), \(p=1\), \(\be,\ac>0\) and \(\te+\be=\si<\ln\ac\), \(\hat\mu=\mu^\si\); then \(P=\bigl(1-\e{-\be}\bigr)\cdot\bigl(\ac-\e\si\bigr)\) and \(Q=0\).
  \item\label{it:ezrp} \(I=\Zb^+\), \(g(y)=\ac\cdot\bigl(\e{\be y}-1\bigr)\), \(p=1\), \(\be,\ac>0\) and \(\si=\te-\be\), \(\hat\mu=\mu^\si\); then \(P=\bigl(\e\be-1\bigr)\cdot\bigl(\e\si+\ac\bigr)\) and \(Q=0\).
  \item\label{it:egzrp} \(I=\Zb\), \(g(y)=\ac\cdot\e{\be y}+\bc\), \(p=1\), \(\be,\ac>0\), \(\bc\ge0\) and \(\ln\bc<\si=\te-\be\), \(\hat\mu=\mu^\si\); then \(P=\bigl(\e\be-1\bigr)\cdot\bigl(\e\si-\bc\bigr)\) and \(Q=0\).
 \end{enumerate}
\end{tm}
We now comment on these cases.
\begin{enumerate}
 \item This is the case of independent walkers. A second class particle in this model is yet another independent particle in the system with no interactions whatsoever, and it indeed sees a flat \(\mu^\te=\hat\mu=\mu^\si\) scenario stationary. No shocks occur in this description, \(P\) and \(Q\) are the constant rates of a single walker.
 \item This is q-TAZRP, its rate can be rewritten as (a constant multiple of) \(g(y)=1-q^y\), and the shock condition is \(\te-\si=\ln q<0\). The jump rate is concave, hence so is the hydrodynamic flux \cite{convex}. A shock therefore jumps upwards in density as we cross it from left to right.
 \item This is the convex counterpart of q-TAZRP, only differs from the next case in having a lower bound on particle occupations. Shocks are negative jumps, accordingly.
 \item This doubly infinite generalised exponential zero range has been known to have random walking shocks since \cite{rwshscp}, in fact its bricklayers counterpart, without the second class particles, was already investigated in \cite{sokvalak}. A slight novelty here is the constant \(\bc\) which has not been noted before in this context. Indeed, without the constant \(\bc\) the marginals have the remarkable property that \(\mu^\te\) is a shift by one of \(\mu^\si\). This property is lost with the introduction of \(\bc\ne0\) (it is obviously false in the above cases when the marginals are both concentrated on \(I=\Zb^+\)). Nevertheless, the random walking shocks structure survives this addition of \(\bc\) in the rates.
\end{enumerate}

Next we discuss the case of multiple shocks, each with a second class particle. The shock measure now takes the form
\begin{equation}
 \un\nu^{\un\si,\un m}:\,=\bigotimes_{i\in\Zb}\nu^{\si_i,m_i},\label{eq:unnudef}
\end{equation}
where the vector \(\un m\) is of non-negative integer components \(m_i\), and
\[
 \nu^{\si,m}(y,\,z)=\left\{
  \begin{aligned}
   &\hat\mu^{\si,m}(y),&&\text{if }z=y+m,\\
   &0,&&\text{if }z\ne y+m
  \end{aligned}
 \right.
\]
is a probability distribution on the one-site coupling space \(I\times I\). It forces exactly \(m_i\) second class particles at site \(i\), and features a yet unknown distribution \(\hat\mu^{\si_i,m_i}\) for \(\om_i=\ze_i-m_i\). We require that \(\mu^{\si,0}=\mu^\si\); the distribution on a site without second class particles agrees with \eqref{eq:mum} of some parameter \(\un\te<\si_i<\bar\te\). 

We say that \(\un\nu^{\un\si,\un m}\) satisfies the \emph{random walking shocks property} if, according to \eqref{eq:cgen}, its time evolution satisfies
\[
 \begin{aligned}
  \frac{\di}{\di t}\un\nu^{\un\si,\un m}S(t)\Bigr|_{t=0}=\sum_{i\in\Zb}&P(m_i,\,m_{i+1},\,\si_i,\,\si_{i+1})\cdot\bigl[\un\nu^{\un\si^{i,i+1},\un m^{i,i+1}}-\un\nu^{\un\si,\un m}\bigr]\\
  +\,&Q(m_i,\,m_{i+1},\,\si_i,\,\si_{i+1})\cdot\bigl[\un\nu^{\un\si^{i+1,i},\un m^{i+1,i}}-\un\nu^{\un\si,\un m}\bigr]
 \end{aligned}
\]
 with positive parameters \(P\) and \(Q\) that now can depend on the number \(m_i\) and \(m_{i+1}\) of second class particles on the sites affected by the change, as well as the parameters \(\si_i\) and \(\si_{i+1}\) on these sites. In \(\un m^{i,i+1}\), the number of second class particles is changed analogously to \eqref{eq:ijch}, while \(\un\si^{i+1,i}\) is left to be defined in the theorem below.

We examine this phenomenon for cases \ref{it:qzrp}, \ref{it:ezrp} and \ref{it:egzrp} from above.
\begin{tm}\label{tm:multi}
 The random walking shocks property holds in each of the cases \ref{it:qzrp}, \ref{it:ezrp} and \ref{it:egzrp} of Theorem \ref{tm:single} if
 \begin{equation}
  \si_{i-1}-\si_i=\left\{
   \begin{aligned}
    &-\be m_i&&\text{(case \ref{it:qzrp}),}\\
    &\be m_i&&\text{(cases \ref{it:ezrp}, \ref{it:egzrp}),}\\
   \end{aligned}
  \right.\qquad\text{and}\qquad\hat\mu^{\si,m}=\mu^\si.\label{eq:mstep}
 \end{equation}
 Then,
 \[
  \si_j^{i,i+1}=\left\{
   \begin{aligned}
    &\si_j,&&\text{for }j\ne i,\\
    &\si_i-\beta,&&\text{for }j=i\text{ (case \ref{it:qzrp}),}\\
    &\si_i+\beta,&&\text{for }j=i\text{ (cases \ref{it:ezrp}, \ref{it:egzrp}),}
   \end{aligned}
  \right.
 \]
 \[
  Q(m_i,\,m_{i+1},\,\si_i,\,\si_{i+1})=0,\qquad P(m_i,\,m_{i+1},\,\si_i,\,\si_{i+1})=\left\{
   \begin{aligned}
    &\bigl(1-\e{-\be m_i}\bigr)\cdot\bigl(\ac-\e{\si_i}\bigr),&&\text{case \ref{it:qzrp},}\\
    &\bigl(\e{\be m_i}-1\bigr)\cdot\bigl(\e{\si_i}+\ac\bigr),&&\text{case \ref{it:ezrp},}\\
    &\bigl(\e{\be m_i}-1\bigr)\cdot\bigl(\e{\si_i}-\bc\bigr),&&\text{case \ref{it:egzrp}.}\\
   \end{aligned}
  \right.
 \]
Notice that all \(\si_i\)'s must be bounded from above by \(\ln\ac\) in case \ref{it:qzrp}, and from below by \(\ln\bc\) in case \ref{it:egzrp} for the distributions to make sense.
\end{tm}
The interpretation is again that the product distribution \eqref{eq:unnudef} evolves into a linear combination of similar product measures with changed \(\un\si\) and \(\un m\) parameters, and the coefficients can be given an interacting random walks interpretation. To see this, notice that the steps \(\un m^{i,i+1}\) and \(\un\si^{i,i+1}\) keep \eqref{eq:mstep} for all times. The walkers are the second class particles themselves, jumping from \(i\) to \(i+1\) with rate \(P(m_i,\,m_{i+1},\,\si_i,\,\si_{i+1})\). This is clearly a generalization of the random walking shock property \eqref{eq:main} for a single shock.

The rates \(P\) give insights on the interaction between shocks. Notice first that the form we obtained for \(P\) is rather similar to the rates \(g\) of the model itself. This is something one has seen in \cite{rwshscp} to some extent in both ASEP and EBLP, and it repeats here again. Multiple walkers are allowed to occupy the same site. In this case, we see stronger interactions than with EBLP before: except for \(\bc=0\) in case \ref{it:egzrp}, the rate \(P\) of jump from a site \(i\) will not coincide with the sum of the rates of individual shock-walkers. This is due to the convexity of the exponential function.

To see that shock-walkers stay in tight distance to each other, we can repeat the argument for Lemma 4.2 of \cite{sokvalak}. Namely, first label the walkers in increasing order from left to right (and run the label dynamics to keep this order). Then notice that two neighbouring walkers, if both are alone on their sites \(i<j\), have higher rate to decrease their distance then to increase it. Adding higher-labeled walkers on site \(j\) blocks the front walker due to the labeling dynamics, while adding lower-labeled ones on site \(i\) makes \(m_i\), hence the rate of decreasing the distance, even higher. Thus, each inter-walker distance is stochastically bounded by independent Geometric distributions.

Thinking of hydrodynamics, it is natural to expect that our shock-walkers have the Rankine-Hugoniot velocity \eqref{eq:rhv} as their mean velocity. We verify this directly in certain cases:
\begin{pr}\label{tm:rh}
 The Rankine-Hugoniot formula \eqref{eq:rhv} holds for the velocity in single shock cases of Theorem \ref{tm:single}, and the two-shock cases of Theorem \ref{tm:multi}.
\end{pr}
Due to the more complicated interaction of shocks, we do not see an easy way to verify this for more than two shocks, as has been done for the examples of \cite{rwshscp}.

There is a well established connection between exclusion and zero range processes that works via mapping inter-particle distances of exclusion into zero range models. The classical ASEP maps to constant rate asymmetric zero range, and q-TASEP maps to q-TAZRP this way. Notice that random walking shocks in the form discussed here are not known to appear in constant rate zero range, nor in q-TASEP. The mapping of course gives something, but this seems rather unnatural as the second class particles of the respective models do not get maped into each other either. The random walking shocks of ASEP, with \cite{rwshscp} or without \cite{qse} the second class particle, give a strange statement for constant rate zero range that the authors can best imagine by simply undoing the mapping back to ASEP, which does not do much good to the state of the art. The random walking shocks discovered here for q-TAZRP also mean something in q-TASEP, but this also seems to be best imagined by undoing the mapping back to q-TAZRP. In fact, it is via the mapping that already flat, constant density stationary distributions of q-TASEP are best understood.

\section{Proofs}\label{sc:pf}

We build on \cite{rwshscp}, and cite various steps and formulas from there across our proofs.
\begin{proof}[Proof of Theorem \ref{tm:single}]
 We plug \eqref{eq:pqdef} in (49) and (50) of \cite{rwshscp}, with marginals \eqref{eq:mum} but no assumption on \(\hat\mu\). Also, we rename \(x:\,=\om_{-1}\), \(y:\,=\om_0\), \(z:\,=\om_1\) from there. Then checking the existence of random walking shocks amounts to ensuring that
 \[
  \begin{aligned}
   B+D&=(1-p)\bigl[g(y+1)-g(y)]\cdot\frac{\mu^\te(x)\hat\mu(y)}{\hat\mu(x)\mu^\si(y)},\\
   C+E&=p\bigl[g(y+1)-g(y)]\cdot\frac{\hat\mu(y)\mu^\si(z)}{\mu^\te(y)\hat\mu(z)}
  \end{aligned}
 \]
 are each constant as functions of \(y\) and \(x\), resp.\ \(z\). As \(p>0\) in all cases, the second line tells us that \(\hat\mu=\mu^\si\).
 
 If \(p<1\) then this implies in the first line that \(g\) is an affine function, and since it cannot go negative, the range of occupations cannot be doubly infinite, that is \(I=\Zb^+\). Furthermore, this line also tells us that \(\mu^\te=\hat\mu\), which in turn is equal to \(\mu^\si\), and this gives case \ref{it:iw} of the theorem.
 
 If \(p=1\) then the first line of the display is empty, and in the second line
 \[
  \frac{\hat\mu(y)}{\mu^\te(y)}=\frac{\mu^\si(y)}{\mu^\te(y)}=\frac{Z(\te)}{Z(\si)}\cdot\e{(\si-\te)y}
 \]
 which gives
 \[
  g(y+1)-g(y)=\text{const}\cdot\e{(\te-\si)y}.
 \]
 The general solution is
 \begin{equation}
  g(y)=\hat\ac\cdot\e{(\te-\si)y}+\bc,\label{eq:ezrp}
 \end{equation}
 where the only restriction so far is that \(g(0)=0\) if \(I=\Zb^+\) and \(g(y)\ge0\) for all \(y\in I\) must hold. For \(I=\Zb\) this means \(\hat\ac>0\) and \(\bc\ge0\), and attractivity requires \(\te>\si\) as well. Moreover, \(\te>\si>\un\te=\ln\bc\). This is case \ref{it:egzrp}.
 
 When \(I=\Zb^+\), \(\hat\ac>0\) requires again \(\te>\si\), and we need \(\bc=-\hat\ac\) for \(g(0)=0\). This is case \ref{it:ezrp}. However, we can also have \(\hat\ac<0\) with \(\te<\si<\bar\te=\ln\bc\) and \(\bc=-\hat\ac\). This is case \ref{it:qzrp} after a sign change in the constant, \(\ac=-\hat\ac\).
 
 We next check that \(A\) of (52) in \cite{rwshscp} is also independent of \(x\), \(y\) and \(z\). We only indicate how it is done in the linear case \ref{it:iw}, and concentrate on the other cases. For these latter,
 \[
  \begin{aligned}
   A&=g(\om_{b+1})-g(\om_{a-1})\\
   &\quad+g(x+1)\cdot\frac{\e\te\cdot\mu^\si(y-1)}{g(x+1)\mu^\si(y)}-g(y)+g(y+1)\cdot\frac{\mu^\si(y+1)\cdot\e{-\si}\cdot g(z)}{\mu^\si(y)}-g(y+1)+g(y)-g(z)\\
   &=g(\om_{b+1})-g(\om_{a-1})+g(y)\cdot\e{\te-\si}-g(y)+g(z)-g(y+1)+g(y)-g(z)\\
   &=g(\om_{b+1})-g(\om_{a-1})+g(y)\cdot\e{\te-\si}-g(y+1).
  \end{aligned}
 \]
For the linear case \ref{it:iw} the above covers the right jumps part of the dynamics. The exponential vanishes and the increment is constant. The left jump parts work similarly. For all other cases we now use \eqref{eq:ezrp}:
 \begin{equation}
  A=g(\om_{b+1})-g(\om_{a-1})+(\e{\te-\si}-1)\bc.\label{eq:asi}
 \end{equation}
 This being a constant, \eqref{eq:ezrp} with \(p=1\) does give models with random walking shocks. The left and right jump rates of the shock are, respectively (cf.\ \cite{rwshscp})
 \[
  Q=B+D=0,\qquad P=C+E=\hat\ac\bigl(\e{\te-\si}-1\bigr)\cdot\frac{Z(\te)}{Z(\si)}.
 \]
 This latter is further simplified by
 \begin{equation}
  Z(\te)=\sum_{z\in I}\frac{\e{\te z}}{g(z)!}=\sum_{z\in I}\e{(\te-\si)z}\frac{\e{\si z}}{g(z)!}=Z(\si)\cdot\Ev^\si\e{(\te-\si)\om}=Z(\si)\cdot\frac{\Ev^\si g(\om)-\bc}{\hat\ac}=Z(\si)\cdot\frac{\e\si-\bc}{\hat\ac}\label{eq:zsh}
 \end{equation}
 where we last used a general property of expectations within the whole family of misanthrope processes. Thus,
 \begin{equation}
  P=\bigl(\e{\te-\si}-1\bigr)\cdot\bigl(\e\si-\bc\bigr)\label{eq:Pb}
 \end{equation}
which gives the rates as seen in the respective cases of the theorem. Notice that this also coincides with the negative of the expectation of \eqref{eq:asi} w.r.t.\ \(\om_{a-1}\) and \(\om_{b+1}\).
\end{proof}
\begin{proof}[Proof of Proposition \ref{tm:rh}, single shock case]
 We start with the analogue of \eqref{eq:zsh} for expectations. Notice that the function \(\vr(\te)\) can be given in a somewhat explicit expression using q-series or just summing up the below, but the next calculation is sufficient for our purposes. All shock cases of Theorem \ref{tm:single} are captured by \eqref{eq:ezrp}, which we therefore use in the third step.
 \begin{equation}
  \begin{aligned}
   \vr(\te)&=\sum_{z\in I}z\frac{\e{\te z}}{g(z)!Z(\te)}=\frac{\hat\ac}{\e\si-\bc}\sum_{z\in I}z\frac{\e{\si z}}{g(z)!Z(\si)}\cdot\e{(\te-\si)z}\\
   &=\frac1{\e\si-\bc}\sum_{z\in I}z\frac{\e{\si z}}{g(z-1)!Z(\si)}-\frac\bc{\e\si-\bc}\sum_{z\in I}z\frac{\e{\si z}}{g(z)!Z(\si)}\\
   &=\frac{\e\si}{\e\si-\bc}\sum_{y\in I}(y+1)\frac{\e{\si y}}{g(y)!Z(\si)}-\frac\bc{\e\si-\bc}\sum_{z\in I}z\frac{\e{\si z}}{g(z)!Z(\si)}\\
   &=\vr(\si)+\frac{\e\si}{\e\si-\bc}.
  \end{aligned}\label{eq:rhos}
 \end{equation}
 The change \(y=z+1\) of summation variable was possible due to the first term being zero if \(I=\Zb^+\), and summability when \(I=\Zb\). With this result, the Rankine-Hugoniot velocity reads
 \[
  V=\frac{\Hc\bigl(\vr(\si)\bigr)-\Hc\bigl(\vr(\te)\bigr)}{\vr(\si)-\vr(\te)}=-\frac{\e\si-\e\te}{\e\si}\cdot\bigl(\e\si-\bc\bigr)=P
 \]
 from \eqref{eq:Pb}, as required.
\end{proof}
\begin{proof}[Proof of Theorem \ref{tm:multi}]
 We need to check \cite[(54), (55), (56)]{rwshscp} with the assumptions we have. Equation (55) is trivially zero. Checking (56), we'll use \(\mp\) as \(-\) for case \ref{it:qzrp} and \(+\) for cases \ref{it:ezrp} and \ref{it:egzrp}. With this notation we rewrite
 \[
  g(y)=\mp\ac\cdot\e{\mp\be y}\mp\bc
 \]
 which now conforms the statement of Theorem \ref{tm:single}, with \(\bc=-\ac\) for cases \ref{it:qzrp} and \ref{it:ezrp}. We have, from \cite[(56)]{rwshscp},
 \[
  C_i+E_i=\bigl[g(\om_i+m_i)-g(\om_i)\bigr]\cdot\frac{\mu^{\si_i}(\om_i)\mu^{\si_{i+1}}(\om_{i+1})}{\mu^{\si_i\mp\be}(\om_i)\mu^{\si_{i+1}}(\om_{i+1})}=\bigl[g(\om_i+m_i)-g(\om_i)\bigr]\cdot\frac{\e{\pm\be\om_i}Z(\si_i\mp\be)}{Z(\si_i)}.
 \]
 Repeating \eqref{eq:zsh} in this scenario,
 \[
  Z(\si_i\mp\be)=\sum_{z\in I}^\infty\e{\mp\be z}\frac{\e{\si_iz}}{g(z)!}=Z(\si_i)\cdot\Ev^{\si_i}\e{\mp\be\om}=\mp Z(\si_i)\cdot\frac{\Ev^{\si_i}g(\om)\pm\bc}\ac=\mp Z(\si_i)\cdot\frac{\e{\si_i}\pm\bc}\ac,
 \]
 hence
 \[
  C_i+E_i=\ac\e{\mp\be\om_i}\bigl(\e{\mp\be m_i}-1)\e{\pm\be\om_i}\cdot\frac{\e{\si_i}\pm\bc}\ac=\bigl(\e{\mp\be m_i}-1)\cdot\bigl(\e{\si_i}\pm\bc\bigr).
 \]
 This gives the \(P\) rates in the statement of the theorem, as soon as we can see that \(A_i\) of \cite[(54)]{rwshscp} does not depend on \(\om_a,\,\dots,\,\om_b\) and its mean w.r.t.\ \(\om_{a-1}\) and \(\om_{b+1}\) exactly cancels the previous display under the summation of \cite[(53)]{rwshscp}. To this order, we proceed with calculating \cite[(54)]{rwshscp} for our setting.
 \[
  \begin{aligned}
   A_i&=g(\om_i+1)\cdot\frac{\mu^{\si_i}(\om_i+1)\mu^{\si_{i+1}}(\om_{i+1}-1)}{\mu^{\si_i}(\om_i)\mu^{\si_{i+1}}(\om_{i+1})}-g(\om_i)+g(\om_i)-g(\om_i+m_i)\\
   &=\e{\si_i-\si_{i+1}}g(\om_{i+1})-g(\om_i+m_i)\\
   &=\e{\mp\be m_{i+1}}g(\om_{i+1})-g(\om_i+m_i)\\
   &=\mp\ac\e{\mp\be(\om_{i+1}+m_{i+1})}\pm\ac\e{\mp\be(\om_i+m_i)}\mp\e{\mp\be m_{i+1}}\bc\pm\bc.
  \end{aligned}
 \]
 Performing the sum in \cite[(53)]{rwshscp}, the first two terms telescope:
 \[
  \sum_{i=a-1}^bA_i=\mp\ac\e{\mp\be(\om_{b+1}+m_{b+1})}\pm\ac\e{\mp\be(\om_{a-1}+m_{a-1})}\mp\bc\sum_{i=a-1}^b\bigl(\e{\mp\be m_{i+1}}-1\bigr).
 \]
At this point we see that indeed the sum does not depend on \(\om_a,\,\dots,\,\om_b\), as required. By assumption, see \cite{rwshscp}, all second class particles are in the volume \(a\dots b\) of dependence for the test function \(\vp\). Hence \(m_{a-1}=m_{b+1}=0\). By \(\Ev^\te g(\om)=\e\te\), we can calculate the mean
 \[
  \Ev\sum_{i=a-1}^bA_i=\Ev g(\om_{b+1})\pm\bc-\Ev g(\om_{a-1})\mp\bc\mp\bc\sum_{i=a-1}^b\bigl(\e{\mp\be m_{i+1}}-1\bigr)=\e{\si_{b+1}}-\e{\si_{a-1}}\mp\bc\sum_{i=a-1}^b\bigl(\e{\mp\be m_{i+1}}-1\bigr).
 \]
All that is left is to check that \(\sum_{i=a-1}^b(C_i+E_i)\) cancels with this. Notice that \(\si_i\mp\be m_i=\si_{i-1}\). Thus parts of \(C_i+E_i\) rewrite into a telescopic expression:
 \[
  \sum_{i=a-1}^b(C_i+E_i)=\sum_{i=a-1}^b\Bigl(\e{\si_{i-1}}-\e{\si_i}\pm\bc\bigl(\e{\mp\be m_i}-1\bigr)\Bigr)=\e{\si_{a-2}}-\e{\si_b}\pm\bc\sum_{i=a-1}^b\bigl(\e{\mp\be m_i}-1\bigr).
 \]
 Adding the two previous displays results in
 \[
  \e{\si_{b+1}}-\e{\si_b}+\e{\si_{a-2}}-\e{\si_{a-1}}\pm\bc\bigl(\e{\mp\be m_{a-1}}-\e{\mp\be m_{b+1}}\bigr).
 \]
But this is zero due to \(m_{a-1}=m_{b+1}=0\) and \eqref{eq:mstep}.
\end{proof}
\begin{proof}[Proof of Proposition \ref{tm:rh}, two shocks]
 The distance between the two walkers is a birth and death chain with rates, to be specified later,
 \[
  \rnup\text{ for }0\to1,\qquad \rup\text{ for }d\to d+1\quad(d>0),\qquad \rdn\text{ for }d\to d-1\quad(d>0).
 \]
 Assuming \(\rup<\rdn\), the stationary distribution gives probability
 \[
  \frac{\rdn-\rup}{\rdn+\rnup-\rup}\text{ on }0,\qquad\text{and}\qquad\frac{\rnup}{\rdn+\rnup-\rup}\text{ on }\{1,\,2,\,\dots\}.
 \]
 When the chain is at 0, the two walkers coincide, and we have a single rate with \(m=2\) for one of them to move. This coincides with \(\rnup\). Otherwise both will have rates with \(m=1\) to move; the one behind has rate \(\rdn\) while the one ahead has rate \(\rup\). The center of mass of the two walkers moves \(\frac12\) units with these rates, hence has velocity
 \[
  C=\frac\rnup2\cdot\frac{\rdn-\rup}{\rdn+\rnup-\rup}+\frac{\rdn+\rup}2\cdot\frac{\rnup}{\rdn+\rnup-\rup}=\frac{\rnup\rdn}{\rdn+\rnup-\rup}.
 \]
 
 Denote the parameter of zero range on the left by \(\te\). For case \ref{it:qzrp}, with the notations of the statement of Theorem \ref{tm:multi},
 \[
  \rnup=\bigl(1-\e{-2\be}\bigr)\bigl(\ac-\e{\te+2\be}\bigr),\qquad\rup=\bigl(1-\e{-\be}\bigr)\bigl(\ac-\e{\te+2\be}\bigr),\qquad\rdn=\bigl(1-\e{-\be}\bigr)\bigl(\ac-\e{\te+\be}\bigr),
 \]
 and the velocity turns into
 \[
  \begin{aligned}
   C&=\frac{\bigl(1-\e{-2\be}\bigr)\bigl(1-\e{-\be}\bigr)\bigl(\ac-\e{\te+2\be}\bigr)\bigl(\ac-\e{\te+\be}\bigr)}{\bigl(1-\e{-\be}\bigr)\bigl(\e{\te+2\be}-\e{\te+\be}\bigr)+\bigl(1-\e{-2\be}\bigr)\bigl(\ac-\e{\te+2\be}\bigr)}\\
   &=\frac{\bigl(1-\e{-2\be}\bigr)\bigl(\ac-\e{\te+2\be}\bigr)\bigl(\ac-\e{\te+\be}\bigr)}{\bigl(\e{\te+2\be}-\e{\te+\be}\bigr)+\bigl(1+\e{-\be}\bigr)\bigl(\ac-\e{\te+2\be}\bigr)}=\frac{\bigl(1-\e{-2\be}\bigr)\bigl(\ac-\e{\te+2\be}\bigr)\bigl(\ac-\e{\te+\be}\bigr)}{\bigl(1+\e{-\be}\bigr)\ac-2\e{\te+\be}}.
  \end{aligned}
 \]
 The Rankine-Hugoniot velocity reads
 \[
  V=\frac{\Hc\bigl(\vr(\te+2\be)\bigr)-\Hc\bigl(\vr(\te)\bigr)}{\vr(\te+2\be)-\vr(\te)}=\frac{\e{\te+2\be}-\e\te}{\vr(\te+2\be)-\vr(\te+\be)+\vr(\te+\be)-\vr(\te)}.
 \]
 We now return to \eqref{eq:rhos}:
 \[
  \vr(\te+2\be)-\vr(\te+\be)+\vr(\te+\be)-\vr(\te)=\frac{\e{\te+2\be}}{\bc-\e{\te+2\be}}+\frac{\e{\te+\be}}{\bc-\e{\te+\be}},
 \]
 hence 
 \[
  V=\frac{\bigl(\e{\te+2\be}-\e\te\bigr)\bigl(\bc-\e{\te+2\be}\bigr)\bigl(\bc-\e{\te+\be}\bigr)}{\e{\te+2\be}\bigl(\bc-\e{\te+\be}\bigr)+\e{\te+\be}\bigl(\bc-\e{\te+2\be}\bigr)}=C
 \]
 as required once \(\bc=\ac\) is substituted as was done for case \ref{it:qzrp}.

 For case \ref{it:ezrp}, change the sign of \(\ac\), while for case \ref{it:egzrp}, turn \(\ac\) into \(\bc\). In both cases change the sign of \(\be\) as well. In the last step for case \ref{it:ezrp}, substitute \(\bc=-\ac\) as was done in that case. Otherwise the proof is word for word identical.
\end{proof}

\section*{Acknowledgements}

The authors thank anonymous referees for pointing out typos and ways to improve the manuscript. M.\ B.\ also thanks Frank Redig for valuable discussions on duality and its connections to random walking shocks.

\bibliographystyle{plain}
\bibliography{refsmarton}

\end{document}